\documentclass[11pt,a4paper]{article}

\usepackage{epsf,epsfig,amsfonts,amsgen,amsmath,amstext,amsbsy,amsopn,amsthm,cases,listings,color
}
\usepackage{ebezier,eepic}
\usepackage{color}
\usepackage{multirow}
\usepackage{epstopdf}
\usepackage{graphicx}
\usepackage{pgf,tikz}
\usepackage{mathrsfs}
\usepackage[marginal]{footmisc}
\usepackage{enumitem}
\usepackage[titletoc]{appendix}
\usepackage{booktabs}
\usepackage{url}

\usepackage{pgfplots}

\usepackage{subfigure}
\usepackage{mathrsfs}
\usetikzlibrary{arrows}

\allowdisplaybreaks[1]

\definecolor{uuuuuu}{rgb}{0.27,0.27,0.27}
\definecolor{sqsqsq}{rgb}{0.1255,0.1255,0.1255}

\newtheorem{definition}{Definition} [section]
\newtheorem{theorem}[definition]{Theorem}
\newtheorem{lemma}[definition]{Lemma}
\newtheorem{proposition}[definition]{Proposition}

\newtheorem{claim}[definition]{Claim}

\newenvironment{pf}{\noindent{\bf Proof}}

\begin{document}
\title{\bf\Large A note on explicit constructions of designs}

\date{\today}

\author{
Xizhi Liu
\thanks{Department of Mathematics, Statistics, and Computer Science, University of Illinois, Chicago, IL, 60607 USA.
email: xliu246@uic.edu.
Research partially supported by NSF awards DMS-1763317 and DMS-1952767.
}
\and
Dhruv Mubayi
\thanks{Department of Mathematics, Statistics, and Computer Science, University of Illinois, Chicago, IL, 60607 USA.
email: mubayi@uic.edu.
Research partially supported by NSF awards DMS-1763317 and DMS-1952767.}
}
\maketitle
%\footnote{footnote}
%%%%%%%%%%%%%%%%%%%%%%%%%%%%%%%%%%%%%%%%%%%%%%%%%
\begin{abstract}
An $(n,r,s)$-system is an $r$-uniform hypergraph on $n$ vertices such that every pair of edges has an intersection of size less than $s$.
Using probabilistic arguments, R\"{o}dl and \v{S}i\v{n}ajov\'{a} showed that for all fixed integers $r> s \ge 2$,
there exists an $(n,r,s)$-system with independence number $O\left(n^{1-\delta+o(1)}\right)$ 
for some optimal constant $\delta >0$ only related to $r$ and $s$.
We show that for certain pairs $(r,s)$ with $s\le r/2$ there exists an explicit construction of an $(n,r,s)$-system with independence number
$O\left(n^{1-\epsilon}\right)$, where $\epsilon > 0$ is an absolute constant only related to $r$ and $s$.
Previously this was known only for $s>r/2$ by results of Chattopadhyay and Goodman

\end{abstract}
%%%%%%%%%%%%%%%%%%%%%%%%%%%%%%%%%%%%%%%%%%%%%%%%%

\section{Introduction}\label{SEC:introduction}
For a finite set $V$ and a positive integer $r$ denote by $\binom{V}{r}$ the collection of all $r$-subsets of $V$.
An $r$-uniform hypergraph ($r$-graph) $\mathcal{H}$ is a family of $r$-subsets of finite set
which is called the vertex set of $\mathcal{H}$ and is denoted by $V(\mathcal{H})$.
A set $I \subset V(\mathcal{H})$ is {\em independent} in $\mathcal{H}$ if it contains no edge of $\mathcal{H}$.
The {\em independence number} of $\mathcal{H}$, denoted by $\alpha(\mathcal{H})$, is the maximum size of an independent set in $\mathcal{H}$.

For integers $n\ge r\ge  s \ge 1$ an {\em $(n,r,s)$-system} (also called design) is
an $r$-graph on $n$ vertices such that every pair of edges has an intersection of size less than $s$.
R\"{o}dl and \v{S}i\v{n}ajov\'{a} \cite{RS94} proved a lower bound for the independence number of an $(n,r,s)$-system,
and moreover, they showed that there exists an $(n,r,s)$-system whose independence number achieves the lower bound up to a
multiplicative constant factor.

\begin{theorem}[R\"{o}dl-\v{S}i\v{n}ajov\'{a} \cite{RS94}]\label{THM:RS-independence-number-design}
For fixed integers $r> s\ge 2$ there exists a constant $c = c(r,s)$ such that every $(n,r,s)$-system has
independence number at least $cn^{\frac{r-s}{r-1}}\left(\log n\right)^{\frac{1}{r-1}}$.
Moreover, there exists a constant $C = C(r,s)$ such that there exists an $(n,r,s)$-system with
independence number at most $Cn^{\frac{r-s}{r-1}}\left(\log n\right)^{\frac{1}{r-1}}$ for every integer $n \ge r$.
\end{theorem}

\begin{definition}
For fixed integers $r\ge s \ge 1$
we say there is an explicitly construction of an $(n,r,s)$-system with property $\mathcal{P}$
if there exists an algorithm $\mathcal{A}$ such that for every integer $n$ as input,
$\mathcal{A}$ runs in time ${\rm poly}(n)$ and outputs an $(n,r,s)$-system with property $\mathcal{P}$.
\end{definition}

Explicit constructions of $(n,r,s)$-systems with certain properties are very useful in theoretical computer science.
For example, in the seminal work of Nisan and Wigderson \cite{NW94},  dense $(n,r,s)$-systems
are used to construct pseudorandom generators (PRGs) (see also \cite{Tre01,RRV02} for more applications).
More recently, explicit constructions of $(n,r,s)$-systems with small independence number were used
to construct extractors for adversarial sources \cite{CGGL20,CG20}.

In this note, we focus on the explicit constructions of $(n,r,s)$-systems with small independence number.
R\"{o}dl and \v{S}i\v{n}ajov\'{a}'s proof of the existence of an $(n,r,s)$-system with small independence number
uses the Lov\'{a}sz local lemma, and hence it does not provide an explicit way to construct them.
Perhaps the first explicit construction of an $(n,3,2)$-system (also called a Steiner triple system)
with independence number $O(n^{1-\epsilon})$ for some absolute constant $\epsilon>0$ is due to
Chattopadhyay, Goodman, Goyal, and Li \cite{CGGL20}. Their proof uses results about cap sets (see \cite{CLP17,EG17}).

\begin{theorem}[Chattopadhyay-Goodman-Goyal-Li \cite{CGGL20}]\label{THM:CGGL-construction-3-2-system}
There exists a constant $C \ge 1$ such that for every integer $n\ge 3$
there exists an explicit construction of an $(n,3,2)$-system with independence number at most $Cn^{0.9228}$.
\end{theorem}

%The proof of Theorem~\ref{THM:CGGL-construction-3-2-system} uses results about cap sets \cite{CLP17,EG17}.

Later, using results about linear codes \cite{Hoc59,BR60} and Sidorenko's recent bounds on the size of sets in $\mathbb{Z}_{2}^{n}$
containing no $r$ elements that sum to zero \cite{Sid18,Sid20},
Chattopadhyay and Goodman \cite{CG20} extended Theorem~\ref{THM:CGGL-construction-3-2-system} to all integers $r> s \ge 2$
with $s \ge \lceil r/2 \rceil$.

\begin{theorem}[Chattopadhyay-Goodman \cite{CG20}]\label{THM:CG-construction-r-s-system}
There exists a constant $C\ge 1$ such that for every integer $s \ge 2$ and every even integer $r> s$ there exists
an explicit construction of an $(n,r,s)$-system with independence number at most $Cr^4 n^{\frac{2(r-s)}{r}}$.
\end{theorem}

{\bf Remark.} For odd $r$ they showed that there exists
an explicit construction of an $(n,r,s)$-system with independence number at most $C(r+1)^4 n^{\frac{2(r+1-s)}{r+1}}$.

Our main results in this note extend Theorem~\ref{THM:CGGL-construction-3-2-system} for certain values of $r$ and $s$
in the range  $s < \lceil r/2 \rceil$ which was not addressed by Theorem\ref{THM:CG-construction-r-s-system}.

Our proof of the first theorem below is based on a recent result about the maximum size of a set in $\mathbb{Z}_{6}^{n}$
that avoids $6$-term arithmetic progressions \cite{PP20}.

\begin{theorem}\label{THM:n-6-2-design-construction}
There exists a constant $C > 0$ such that
for every integer $r\in \{4,5,6\}$ and every integer $n \ge r$
there exists an explicit construction of an $(n,r,2)$-system $\mathcal{H}$ with $\alpha(\mathcal{H}) \le Cn^{0.973}$.
\end{theorem}

Using a lemma about the independence number of the product of two hypergraphs we are able to extend Theorem~\ref{THM:n-6-2-design-construction}
to a wider range of $r$ and $s$.

For every integer $s = 3^{\ell_1}4^{\ell_2}5^{\ell_3}6^{\ell_4}+1$, where $\ell_1,\ell_2,\ell_3,\ell_4 \ge 0$ are integers,
define
\begin{align}
R(s) =
\begin{cases}
6(s-1) & {\rm if}\quad \ell_1 = \ell_2 = \ell_3 = 0 \\
5(s-1) & {\rm if}\quad \ell_1 = \ell_2 = 0 \quad{\rm and}\quad \ell_3\neq 0\\
4(s-1) & {\rm if}\quad \ell_1 = 0 \quad{\rm and}\quad \ell_2\neq 0\\
3(s-1) & {\rm if}\quad \ell_1\neq 0
\end{cases}\notag
\end{align}

\begin{theorem}\label{THM:n-6-2-design-construction-graph-product}
For every integer $s$ of the form $3^{\ell_1}4^{\ell_2}5^{\ell_3}6^{\ell_4}+1$, where $\ell_1,\ell_2,\ell_3,\ell_4 \ge 0$ are integers,
and every integer $r$ satisfying $2s \le r \le R(s)$
there exist constants $C = C(\ell_1,\ell_2,\ell_3,\ell_4), \epsilon = \epsilon(\ell_1,\ell_2,\ell_3,\ell_4) > 0$
such that for every integer $n\ge r$ there exists
an explicit construction of an $(n, r, s)$-system
with independence number at most $Cn^{1-\epsilon}$.
\end{theorem}

The following result focusing on $(n,5,4)$-systems uses a different argument and it improves the bound $O(n^{2/3})$ given by
Theorem~\ref{THM:CG-construction-r-s-system}.

\begin{theorem}\label{THM:5-4-design}
There exists a constant $C>0$ such that for every integer $n\ge 5$ there exists an explicit construction of an
$(n,5,4)$-systems with independence number at most $C n^{\log_{3}2} \le Cn^{0.631}$.
\end{theorem}

We prove Theorems~\ref{THM:n-6-2-design-construction} and~\ref{THM:n-6-2-design-construction-graph-product} in Section~\ref{SEC:n-r-2-design},
and prove Theorem~\ref{THM:5-4-design} in Section~\ref{SEC:5-4-system}.

%%%%%%%%%%%%%%%%%%%%%%%%%%%%%%%%%%%%%%
%%%%%%%%%%%%%%%%%%%%%%%%%%%%%%%%%%%%%%%%
\section{Proofs of Theorems~\ref{THM:n-6-2-design-construction} and~\ref{THM:n-6-2-design-construction-graph-product}}\label{SEC:n-r-2-design}
\subsection{Proof of Theorems~\ref{THM:n-6-2-design-construction}}
Let us first introduce a construction of $r$-graphs based on $r$-term arithmetic progressions ($r$-AP) over $\mathbb{Z}_{r}^{k}$. We do not allow trivial progressions so an $r$-AP has $r$ distinct elements.

{\bf Construction $\mathcal{A}(r,k)$.}
Let $r\ge 3$ and $k\ge 1$ be integers.
The hypergraph $\mathcal{A}(r,k)$ is the $r$-graph with vertex set $V = \mathbb{Z}_{r}^{k}$ and edge set
\begin{align}
\left\{\{v_1,\ldots, v_r\}\in \binom{V}{r}\colon v_1,\ldots, v_r \text{ form an $r$-AP}\right\}. \notag
\end{align}

{\bf Remarks.}
\begin{itemize}
    \item It is clear that $\mathcal{A}(r,k)$ can be constructed in time ${\rm poly}\left(r^k\right)$ for all integers $r,k\ge 1$.
    \item Even though we defined $\mathcal{A}(r,k)$ for all integers $r\ge 3$, in the
    proof of Theorem~\ref{THM:n-6-2-design-construction} we will consider only the case $r=6$.
\end{itemize}

The following easy proposition shows that for every integer $r\ge 3$ the
hypergraph $\mathcal{A}(r,k)$ is {\em linear}, i.e.
every pair of edges has an intersection of size at most one.

\begin{proposition}\label{PROP:A(r,k)-is-n-r-2-system}
Let $r\ge 3$, $k\ge 1$ be integers and $n = r^k$.
Then $\mathcal{A}(r,k)$ is an $(n,r,2)$-system.
\end{proposition}
\begin{proof}[Proof of Proposition~\ref{PROP:A(r,k)-is-n-r-2-system}]
Suppose to the contrary that there exist two distinct edges $E,E' \in \mathcal{H}$ such that $|E\cap E'| \ge 2$.
Assume that $E = \{a, a+d, \ldots, a+ (r-1)d\}$ for some $a,d\in \mathbb{Z}_{r}^{k}$ and $d$ is not the zero vector.
Without loss of generality we may assume that $a\in E\cap E'$
(otherwise we can choose an arbitrary element in $E\cap E'$ and rename it as $a$)
and assume that $E' = \{a, a+id, \ldots, a+(r-1)id\}$ for some integer $i\in [r-1]$.
Since $|E'| = r$, the set $\{0,id \pmod{r},\ldots, (r-1)id \pmod{r}\}$ has size $r$.
Therefore, sets $\{0,id \pmod{r},\ldots, (r-1)id\pmod{r}\}$
and $\{0,1,\ldots, r-1\}$ are identical,
which implies that $E = E'$, a contradiction.
Therefore, $\mathcal{A}(r,k)$ is an $(n,r,2)$-system.
\end{proof}

The next proposition shows that in order to prove Theorem~\ref{THM:n-6-2-design-construction} it suffices
to find an explicit construction of an $(n,6,2)$-system with independence number $O(n^{1-\epsilon})$.

\begin{proposition}\label{PROP:shadow-of-design-is-also-design}
Suppose that there exists an $(n,r,s)$-system with independence number at most $\alpha$.
Then there exists an $(n,r',s)$-system with independence number at most $\alpha$ for every integer $r'\in [s+1,r]$.
\end{proposition}
\begin{proof}[Proof of Proposition~\ref{PROP:shadow-of-design-is-also-design}]
Let $\mathcal{H}$ be an $(n,r,s)$-system with independence number at most $\alpha$. Let $V = V(\mathcal{H})$.
Fix an integer $r'\in [s+1, r]$.
Let the $r'$-graph $\mathcal{H}'$ be obtained from $\mathcal{H}$ in the following way:
for every edge $E\in \mathcal{H}$ replace it by an arbitrary $r'$-set $E' \subset E$.
It is clear that $\mathcal{H}'$ is an $r'$-graph on $V$.
Now suppose that $S \subset V$ is a set of size strictly greater than $\alpha$.
Then, by assumption, there exists an edge $E\in \mathcal{H}$ such that $E\subset S$.
It follows from the definition of $\mathcal{H}'$ that there exists $E'\in \mathcal{H}$ such that $E' \subset E\subset S$.
So, $S$ is not an independent set in $\mathcal{H}'$, which implies that $\alpha(\mathcal{H}') \le \alpha$.
\end{proof}

Another ingredient we need for the proof of Theorem~\ref{THM:n-6-2-design-construction}
is the following result due to Pach and Palincza \cite{PP20}.

\begin{theorem}[Pach-Palincza \cite{PP20}]\label{THM:PP20-6-AP}
Suppose that $k$ is a sufficiently large integer.
Then every set of $\mathbb{Z}_{6}^{k}$ of size greater than $(5.709)^k$ contains a $6$-AP.
\end{theorem}

Now we are ready to prove Theorem~\ref{THM:n-6-2-design-construction}.

\begin{proof}[Proof of Theorem~\ref{THM:n-6-2-design-construction}]
By Proposition~\ref{PROP:shadow-of-design-is-also-design}, it suffices to prove that there exists an
$(n,6,2)$-system $\mathcal{H}$ with $\alpha(\mathcal{H}) = O(n^{0.973})$.

First, for all integers $n$ of the form $6^k$ we let the construction be $\mathcal{H} = \mathcal{A}(6,k)$.
It follows from Proposition~\ref{PROP:A(r,k)-is-n-r-2-system} that $\mathcal{H}$ is an $(n,6,2)$-system.
On the other hand, it follows from the definition of $\mathcal{A}(6,k)$ that a set $S\subset V$ is independent in $\mathcal{A}(6,k)$
iff it does not contain a $6$-AP.
So, by Theorem~\ref{THM:PP20-6-AP}, $|S| \le (5.709)^k$. Therefore, $\alpha(\mathcal{H}) \le (5.709)^k \le n^{0.973}$.

Now suppose that $n$ is not of the form $6^k$. Then let $k$ be the smallest integer such that $n \le 6^k$.
Let $\mathcal{H}$ be any $n$-vertex induced subgraph of $\mathcal{A}(6,k)$.
Then $\alpha(\mathcal{H}) \le \alpha(\mathcal{A}(6,k)) \le (5.709)^k \le 6n^{0.973}$.
\end{proof}%THM

%%%%%%%%%%%%%%%%%%%%%%%%%%%%%%%%
\subsection{Proof of Theorem~\ref{THM:n-6-2-design-construction-graph-product}}
Given two hypergraphs $\mathcal{H}_{1}$ and $\mathcal{H}_{2}$,
the {\em direct product} of $\mathcal{H}_1$ and $\mathcal{H}_2$, denoted by $\mathcal{H}_1\Box \mathcal{H}_2$,
is the hypergraph on $V(\mathcal{H}_1)\times V(\mathcal{H}_2)$ with edge set
\begin{align}
\left\{E_1\times E_2\colon E_1\in\mathcal{H}_1 \text{ and } E_2\in\mathcal{H}_2\right\}, \notag
\end{align}
where $\times$ denotes the usual cartesian product of sets.

{\bf Remark.}
It is clear that there exists an algorithm $\mathcal{A}'$ such that for every input $\left(\mathcal{H}_{1},\mathcal{H}_2\right)$,
$\mathcal{A}'$ runs in time ${\rm poly}\left(|\mathcal{H}_1|\cdot |\mathcal{H}_2|\right)$ and outputs $\mathcal{H}_1\Box \mathcal{H}_2$.

One nice property of the operation defined above is that the direct product of two designs is still a design.

\begin{lemma}\label{LEMMA:product-of-designs}
Suppose that $\mathcal{H}_{1}$ is an $(n_1,r_1,s_1)$-system and $\mathcal{H}_{2}$ is an $(n_2,r_2,s_2)$-system.
Then $\mathcal{H}_1\Box \mathcal{H}_2$ is an $\left(n_1n_2, r_1r_2, \max\{r_1(s_2-1)+1, r_2(s_1-1)+1\}\right)$-system.
\end{lemma}
\begin{proof}[Proof of Lemma~\ref{LEMMA:product-of-designs}]
Let $n = n_1n_2$, $r = r_1r_2$, and $s = \max\{r_1(s_2-1)+1, r_2(s_1-1)+1\}$.
It is clear that $\mathcal{H}_1\Box \mathcal{H}_2$ is an $r$-graph on $n$ vertices.
So it suffices to show that every $s$-set of $V(\mathcal{H}_1)\times V(\mathcal{H}_2)$
is contained in at most one edge in $\mathcal{H}_1\Box \mathcal{H}_2$.

Fix an $s$-set $S\subset V(\mathcal{H}_1)\times V(\mathcal{H}_2)$.
Suppose to the contrary that there exist two distinct edges $E,E' \in \mathcal{H}_1\Box \mathcal{H}_2$
such that $S\subset E\cap E'$.
Assume that $E = E_1\times E_2$ and $E' = E_1'\times E_2'$, where $E_1, E_1' \in \mathcal{H}_1$, $E_2, E_2' \in \mathcal{H}_{2}$,
and $(E_1, E_2)\neq (E_1', E_2')$.
Since $E\cap E' = (E_1\cap E_1')\times (E_2\cap  E_2')$, we have $|E\cap E'| = |E_1\cap E_1'| \times |E_2\cap E_2'|$.
On the other hand, since $(E_1, E_2)\neq (E_1', E_2')$, we have either $E_1 \neq E_1'$ or $E_2\neq E_2'$.
In the former case we have $|E\cap E'| = |E_1\cap E_1'| \times |E_2\cap E_2'| \le r_2(s_1-1)< s$,
and in the latter case we have $|E\cap E'| = |E_1\cap E_1'| \times |E_2\cap E_2'| \le r_1(s_2-1)< s$,
both contradict the assumption that $S\subset E\cap E'$ and $|S| = s$.
\end{proof}

Next, we will show that the independence number of the direct product of two hypergraphs
with small independence number is still relatively small.
To prove this we will use the following bipartite version of the Dependent random choice lemma.
Its proof is basically the same as proofs in \cite{FS11,KR01,AKS03,Sud03},
and for the sack of completeness we include it here.

\begin{lemma}[Dependent random choice, see \cite{FS11,KR01,AKS03,Sud03}]\label{LEMMA:dependent-random-choice-bipartite}
Let $a,m,n_1,n_2,r$ be positive integers and $d_1 \ge 0$ be a real number.
Let $G = G[V_1,V_2]$ be a bipartite graph with $|V_1| = n_1$, $|V_2| = n_2$, and $|G| \ge d_1n_1$.
If there exists a positive integer $t$ such that
\begin{align}
\frac{n_1d_1^t}{n_2^{t}} -  \binom{n_1}{r}\left(\frac{m}{n_2}\right)^t \ge a. \notag
\end{align}
Then there exists a subset $U\subset V(G)$ of size at least $a$ such that every set of
$r$ vertices in $U$ has at least $m$ common neighbors.
\end{lemma}
\begin{proof}[Proof of Lemma~\ref{LEMMA:dependent-random-choice-bipartite}]
Pick a set $T$ of $t$ vertices from $V_2$ uniformly at random with repetition.
Set $A = N(T) \subset V_1$ and let $X$ denote the cardinality of $A$.
By the linearity of expectation,
\begin{align}
\mathbb{E}[X]
= \sum_{v\in V_1}\left(\frac{|N(v)|}{n_2}\right)^t
= n_{2}^{-t}\sum_{v\in V_1}|N(v)|^{t}
\ge n_{2}^{-t}n_1\left(\frac{\sum_{v\in V_1}|N(v)|}{n_1}\right)^t
= \frac{n_1d_1^t}{n_2^{t}}. \notag
\end{align}
Let $Y$ be the random variable counting the number of subsets $S\subset A$
of size $r$ with fewer than $m$ common neighbors.
For a given such subset $S$ the probability that it is a subset of $A$ equals $\left(\frac{|N(S)|}{n_2}\right)^t$.
Since there are at most $\binom{n_1}{r}$ subsets $S\subset V_1$ of size $r$ for which $|N(S)| < m$,
it follows that
\begin{align}
\mathbb{E}[Y]
\le \binom{n_1}{r}\left(\frac{m}{n_2}\right)^t. \notag
\end{align}
By the linearity of expectation,
\begin{align}
\mathbb{E}[X-Y]
\ge \frac{n_1d_1^t}{n_2^{t}} -  \binom{n_1}{r}\left(\frac{m}{n_2}\right)^t
\ge a. \notag
\end{align}
Hence there exists a choice of $T$ for which the corresponding set $A = N(T)$ satisfies $X-Y \ge a$.
Deleting one vertex from each subset $S$ of $A$ of size $r$ with fewer than $m$ common neighbors.
We let $U$ be the remaining subset of $A$.
The set $U$ has at least $X-Y \ge a$ vertices and all subsets of size $r$ have at least $m$ common neighbors.
\end{proof}

The following lemma gives an upper bound for the independence number of the direct product of two hypergraphs.

\begin{lemma}\label{LEMMA:product-of-designs-independence-number}
Suppose that $\mathcal{H}_{1}$ is an $r_1$-graph on $n_1$ vertices with $\alpha(\mathcal{H}_1) < n_1/f(n_1)$
and $\mathcal{H}_{2}$ is an $r_2$-graph on $n_2$ vertices with $\alpha(\mathcal{H}_2) < n_2/g(n_2)$ for some real numbers $f(n_1), g(n_2)\ge 1$.
Then $\mathcal{H}_1\Box \mathcal{H}_{2}$ is an $r_1r_2$-graph on $n_1n_2$ vertices
with $\alpha(\mathcal{H}_1\Box \mathcal{H}_{2}) < n_1n_2/h(n_1,n_2)$,
where $h(n_1,n_2) = \left(f(n_1)/2\right)^{1/t}$ and $t = \lceil \frac{\log\left(n_1^{r_1-1}f(n_1)/r_1!\right)}{\log g(n_2)} \rceil$.
\end{lemma}
\begin{proof}[Proof of Lemma~\ref{LEMMA:product-of-designs-independence-number}]
Let $f = f(n_1)$, $g = g(n_2)$,
$t = \lceil \frac{\log\left(n_1^{r_1-1}f/r_1!\right)}{\log g} \rceil$,
$h = h(n_1,n_2) = \left(f/2\right)^{1/t}$,
$d_1 = n_2/h$,
$m = n_2/g$,
and $a = n_1/f$.
Let $S \subset V(\mathcal{H}_1)\times V(\mathcal{H}_2)$ be a set of size $d_1n_1 = n_1n_2/h$.
Define an auxiliary bipartite graph $G = G[V_1,V_2]$ with $V_1 = V(\mathcal{H}_1)$ and $V_2 = V(\mathcal{H}_2)$,
and $u\in V_1$, $v\in V_2$ are adjacent iff $(u,v)\in S$.
Since
\begin{align}
\frac{n_1d_1^t}{n_2^{t}} -  \binom{n_1}{r_1}\left(\frac{m}{n_2}\right)^t - a
& \ge \frac{n_1}{h^t} - \frac{n_1^{r_1}}{r_1!}\frac{1}{g^{t}} - \frac{n_1}{f}  \notag\\
& = n_1\left(\frac{2}{f} - \frac{n_1^{r_1-1}}{r_1!}\frac{1}{g^{t}} - \frac{1}{f} \right)
\ge n_1\left(\frac{2}{f} - \frac{1}{f} - \frac{1}{f} \right) = 0, \notag
\end{align}
it follows from Lemma~\ref{LEMMA:dependent-random-choice-bipartite} that
there exists a set $U\subset V_1$ of size $n_1/f$
such that every $r_1$-subset of $U$ has at least $n_2/g$ common neighbors.
Since $\alpha(\mathcal{H}_1) < n_1/f$, there exists an $r_1$-subset $E_1 \subset U$ such that $E_1\in \mathcal{H}_1$.
Let $W = N(E_1)$.
Since $|W| \ge n_2/g > \alpha(\mathcal{H}_2)$, there exists an $r_2$-subset $E_2\subset W$ such that $E_2 \in \mathcal{H}_2$.
Since every pair $\{u,v\}$ with $u\in E_1$ and $v\in E_2$ is an edge in $G$,
the set $E_1\times E_2$ is contained in $S$.
This implies that $S$ is not an independent set in $\mathcal{H}_1\Box\mathcal{H}_2$
as it contains the edge $E_1\times E_2\in \mathcal{H}_1\Box\mathcal{H}_2$.
Therefore, $\alpha(\mathcal{H}_1\Box \mathcal{H}_{2}) < n_1n_2/h$.
\end{proof}

Now we are ready to prove Theorem~\ref{THM:n-6-2-design-construction-graph-product}.

\begin{proof}[Proof of Theorem~\ref{THM:n-6-2-design-construction-graph-product}]
We prove this theorem by induction on $\sum_{i\in[4]}\ell_i$.
Theorem~\ref{THM:n-6-2-design-construction} shows that the base case $\sum_{i\in[4]}\ell_i = 0$ holds,
so we may assume that $\sum_{i\in[4]}\ell_i\ge 1$.
Let $s = 3^{\ell_1}4^{\ell_2}5^{\ell_3}6^{\ell_4}+1$,
and let us assume, for the sack of simplicity, that $\ell_1 \ge 1$ (the other cases can be proved using a similar argument).
By Proposition~\ref{PROP:shadow-of-design-is-also-design} it suffices to show there is an explicit construction of
an $(n,R(s),s)$-system with independence number $O(n^{1-\epsilon})$.

Fix $n$ and let
$m = \lceil \sqrt{n} \rceil$, $s_1 = 3^{\ell_1-1}4^{\ell_2}5^{\ell_3}6^{\ell_4}+1$, $r_1 = 3(s_1-1)$.
By the induction hypothesis, there exists an explicit construction $\mathcal{H}_1$ of an $(m,r_1,s_1)$-system
with $\alpha(\mathcal{H}_1) \le C_1 m^{1-\epsilon_1}$, where $C_1>0$ and $\epsilon_1>0$ are constants only related to $r_1$ and $s_1$.
On the other hand, by Theorem~\ref{THM:n-6-2-design-construction}, there exists
an explicit construction $\mathcal{H}_2$ of an $(m,3,2)$-system with $\alpha(\mathcal{H}_2) \le C_2 m^{1-\epsilon_2}$,
where $C_2>0$ and $\epsilon_2>0$ are absolute constants.
Let $C = C(C_1,C_2,\epsilon_1, \epsilon_2)>0$ be a sufficiently large constant,
$\epsilon =  \epsilon(C_1,C_2,\epsilon_1, \epsilon_2)>0$ be a sufficiently small constant
($C$ and $\epsilon$ can be determined from the proof below),
and let $\mathcal{H}_3 = \mathcal{H}_1\Box \mathcal{H}_{2}$.
Then by Lemma~\ref{LEMMA:product-of-designs}, $\mathcal{H}_3$ is an $(m^2, 3(s-1), s)$-system.
Applying Lemma~\ref{LEMMA:product-of-designs-independence-number} to $\mathcal{H}_3$ with
$f(m) = m^{\epsilon_1}/C_1$, $g(m) = m^{\epsilon_2}/C_2$ we obtain
$t = \lceil \frac{\log\left(r_1!C_1\right)}{\log C_2}\frac{r_1-1+\epsilon_1}{\epsilon_2}\rceil$,
$h(m,m) = \left(m^{\epsilon_1}/2C_1\right)^{1/t}$, and
$\alpha(\mathcal{H}_3) \le m^2/h(m,m) \le Cn^{1-\epsilon}$
(we can choose $C>0$ to be sufficiently large and $\epsilon>0$ to be sufficiently small
such that the last inequality holds for all integers $n$).
Finally, to obtain an explicit construction of an $(n,3(s-1),s)$-system with independent number at most $C n^{1-\epsilon}$
one just needs to take any $n$-vertex induced subgraph of $\mathcal{H}_3$.
\end{proof}
%%%%%%%%%%%%%%%%%%%%%%%%%%%%%%%%%%%%%%%%%%%%%%%

\section{$(n,5,4)$-systems}\label{SEC:5-4-system}
We prove Theorem~\ref{THM:5-4-design} in this section.

\begin{proof}[Proof of Theorem~\ref{THM:5-4-design}]
We will show that it suffices to choose $C = 21$.
Similar to the proof of Theorem~\ref{THM:n-6-2-design-construction} it suffices to show
an explicit construction of an $(n,5,4)$-system with independence number at most
$7n^{\log_{3}2}-\frac{\sqrt{2}}{2-\sqrt{3}}n^{1/2}$ (this is slightly stronger that what we need)
for all integers $n$ of the form $3^k$, and we will prove it by induction on $k$.

For $k\le 3$ we have $7\left(3^k\right)^{\log_3 2}- \frac{\sqrt{2}}{2-\sqrt{3}}3^{k/2} \ge 3^k$,
so we may assume that $k \ge 4$ and focus on the induction step.
Fix an integer $k$ and let $\mathcal{H}_k$ be a $(3^k,5,4)$-system with
$\alpha(\mathcal{H}_k) \le 7\left(3^k\right)^{\log_3 2}-\frac{\sqrt{2}}{2-\sqrt{3}}3^{k/2} = 7\cdot 2^k-\frac{\sqrt{2}}{2-\sqrt{3}}3^{k/2}$.
Let $\ell \in \mathbb{N}$ such that $2^{\ell} \ge 3^k > 2^{\ell-1}$.
Let $U_1, U_2, U_3$ be three pairwise disjoint copies of $\mathbb{F}_{2^{\ell}}\setminus \{0\}$,
where $\mathbb{F}_{2^{\ell}}$\footnote{It is clear that $\mathbb{F}_{2^{\ell}}$ can be constructed in time ${\rm poly}(2^{\ell})$ for every integer $\ell\ge 1$.}
is the finite field of order $2^{\ell}$ with characteristic $2$.
For $i\in [3]$ let $\psi_i \colon V(\mathcal{H}_k) \to U_i$ be an injection and let $V_i = \psi_{i}(V(\mathcal{H}_{k}))$.
Let $\mathcal{H}_{k+1}$ be the $5$-graph on $V = V_1\cup V_2\cup V_3$ whose edge set is
\begin{align}
\mathcal{H}_{k+1} & = \left\{\{a_1,b_1,a_2,b_2,c\}\in \binom{V}{5}\colon
                         a_1,b_1\in V_1, a_2, b_2 \in V_2, c\in V_3, a_1+b_1\cdot c = a_2+b_2\cdot c\right\} \notag\\
& \qquad \cup \left(\bigcup_{i\in [3]}\psi_i\left(\mathcal{H}_{k}\right)\right). \notag
\end{align}

\begin{claim}\label{CLAIM:5-4-system-H-k+1-A}
$\mathcal{H}_{k+1}$ is a $(3^{k+1},5,4)$-system.
\end{claim}
\begin{proof}[Proof of Claim~\ref{CLAIM:5-4-system-H-k+1-A}]
Let $S = \{a,b,c,d\} \subset V_1 \cup V_2 \cup V_3$ be a set of size $4$.
It is clear that if $|S\cap V_i| \ge 3$ for some $i\in [3]$ or $|S\cap V_3| \ge 2$,
then $S$ can be contained in at most one edge of $\mathcal{H}_{k+1}$.
So we may assume that $|S\cap V_1|, |S\cap V_2| \le 2$ and $|S\cap V_3| \le 1$.

Suppose that $|S \cap V_1| = |S\cap V_2| = 2$,
and without loss of generality we may assume that $S\cap V_1 = \{a,b\}$ and $S\cap V_2 = \{c,d\}$.
By the definition of $\mathcal{H}_{k+1}$, every vertex $e\in V_3$ that satisfies $\{a,b,c,d,e\}\in \mathcal{H}_{k+1}$
must satisfy $a+ c\cdot e = b + d\cdot e$ or $a+ d\cdot e = d+ c\cdot e$.
Since both equations yield $e = \frac{a+b}{c+d}$ (here we used the fact that $x-y = x+y$ holds for all $x,y\in \mathbb{F}_{2^{\ell}}$),
such vertex $e$ is unique. Therefore, $S$ is contained in at most one edge in $\mathcal{H}_{k+1}$.

Suppose that $|S \cap V_1| = 2$ and $|S\cap V_2| = |S\cap V_3| = 1$.
Without loss of generality we may assume that $S\cap V_1 = \{a,b\}$, $S\cap V_2 = \{c\}$, and $S\cap V_3= \{d\}$.
It is easy to see that every vertex $e \in V$ that satisfies $\{a,b,c,d,e\}\in \mathcal{H}_{k+1}$ must satisfy
\begin{itemize}
\item $e\in V_2$, and
\item $a+ c\cdot d = b + e\cdot d$ or $a+ e\cdot d = b+ c\cdot d$.
\end{itemize}
Since both $a+ c\cdot d = b + e\cdot d$ and $a+ e\cdot d = b+ c\cdot d$ imply $e = \frac{a+b}{d}+ c$
(here we used the fact that $x-y = x+y$ holds for all $x,y\in \mathbb{F}_{2^{\ell}}$ again),
such vertex $e$ is unique. Therefore, $S$ is contained in at most one edge in $\mathcal{H}_{k+1}$.

By symmetry, for the other cases one can show that $S$ is contained in at most one edge in $\mathcal{H}_{k+1}$.
Therefore, $\mathcal{H}_{k+1}$ is a $(3^{k+1},5,4)$-system.
\end{proof}

\begin{claim}\label{CLAIM:5-4-system-H-k+1-B}
$\alpha(\mathcal{H}_{k+1})
\le 2\left(7\cdot 2^k - \frac{\sqrt{2}}{2-\sqrt{3}}3^{k/2}\right) + \sqrt{2}\cdot 3^{k/2}.$
\end{claim}
\begin{proof}[Proof of Claim~\ref{CLAIM:5-4-system-H-k+1-B}]
Suppose to the contrary that there exists an independent set $S \subset V$
of size greater than $2\left(7\cdot 2^k - \frac{\sqrt{2}}{2-\sqrt{3}}3^{k/2}\right) + \sqrt{2}\cdot 3^{k/2}$.
Let $S_i = S\subset V_i$ and $s_i = |S_i \cap V|$ for $i\in [3]$.
Since $S$ is independent in $\mathcal{H}_{k+1}$, $S_i$ must be independent in $\psi_i(\mathcal{H}_{k})$.
Therefore, $s_i \le \alpha(\mathcal{H}_k) \le 7\cdot 2^k - \frac{\sqrt{2}}{2-\sqrt{3}}3^{k/2}$
for $i\in [3]$ and consequently, $s_i > \sqrt{2}\cdot 3^{k/2}$ for $i\in [3]$.
Moreover, we have $s_1 + s_2 > 7\cdot 2^k - \frac{\sqrt{2}}{2-\sqrt{3}}3^{k/2} + \sqrt{2}\cdot 3^{k/2}$ and hence,
$$s_1 \cdot s_2
> \left(7\cdot 2^k - \frac{\sqrt{2}}{2-\sqrt{3}}3^{k/2}\right)\cdot \sqrt{2}\cdot 3^{k/2}
\ge \sqrt{2}\left(7-\frac{\sqrt{2}}{2-\sqrt{3}}\right)\cdot 2^k\cdot 3^{k/2}
\ge 2\cdot 3^k \ge 2^{\ell}.$$
Fix $z \in S_3$.
Since $s_1 s_2 > 2^{\ell}$, by the Pigeonhole principle, there exists distinct elements
$(a_1, b_1), (a_2,b_2)\in S_1 \times S_2$ such that $a_1 + b_1 \cdot z = a_2 + b_2 \cdot z$.
It is easy to see that $a_1 \neq a_2$ and $b_1 \neq b_2$ since otherwise the equation $a_1 + b_1 \cdot z = a_2 + b_2 \cdot z$
would imply $(a_1,b_1) = (a_2,b_2)$, a contradiction.
Therefore, $|\{a_1,a_2,b_1,b_2,z\}| = 5$ and hence, $\{a_1,a_2,b_1,b_2,z\} \in \mathcal{H}_{k+1}$.
However, this implies that $S$ contains an edge in $\mathcal{H}_{k+1}$, a contradiction.
\end{proof}%CLAIM

Claim~\ref{CLAIM:5-4-system-H-k+1-B} shows that
$$\alpha(\mathcal{H}_{k+1})
\le 2\left(7\cdot 2^k - \frac{\sqrt{2}}{2-\sqrt{3}}3^{k/2}\right) + \sqrt{2}\cdot 3^{k/2}
=7\cdot 2^{k+1} - \frac{\sqrt{2}}{2-\sqrt{3}}3^{(k+1)/2}.$$
This completes the proof of the induction step.

Notice that given $\mathcal{H}_{k}$ the $r$-graph $\mathcal{H}_{k+1}$ can be constructed in time
${\rm poly}(|\mathcal{H}_{k}|) + {\rm poly}(2^{\ell}) = {\rm poly}(3^k)$.
So for every integer $k \ge 1$ the $r$-graph $\mathcal{H}_{k}$ can be constructed in time ${\rm poly}(3^k)$.
\end{proof}%THM

%%%%%%%%%%%%%%%%%%%%%%%%%%%%%%%%%%%%%%%%%%%%%%%%%
\bibliographystyle{abbrv}
\bibliography{explicit_design}
\end{document}